\documentclass{amsproc}
\usepackage{graphicx}
\usepackage{amssymb}
\usepackage{epstopdf}
\usepackage{amsmath,amsthm,amscd,amssymb}
\usepackage{latexsym}
\usepackage[colorlinks,citecolor=red,pagebackref,hypertexnames=false]{hyperref}
\usepackage{geometry}                % See geometry.pdf to learn the layout options. There are lots.
\numberwithin{equation}{section}

\theoremstyle{plain}
\newtheorem{theorem}{Theorem}[section]
\newtheorem{lemma}[theorem]{Lemma}

\newtheorem{proposition}[theorem]{Proposition}

\theoremstyle{definition}
\newtheorem{definition}[theorem]{Definition}

\newtheorem{problem}[theorem]{Problem}

\newtheorem{case[theorem]}{Case}

\theoremstyle{remark}
\newtheorem{remark}[theorem]{Remark}

\numberwithin{equation}{section}

\begin{document}

\title{\parbox{14cm}{\centering{Sharpness of Falconer's estimate in continuous and arithmetic settings, geometric incidence theorems and distribution of lattice points in convex domains}}}


\author{Alex Iosevich and Steven Senger}

\begin{abstract} In the paper introducing the celebrated Falconer distance problem, Falconer proved that the Lebesgue measure of the distance set is positive, provided that the Hausdorff dimension of the underlying set is greater than $\frac{d+1}{2}$. His result is based on the estimate 
\begin{equation} \label{key} \mu \times \mu \{(x,y): 1 \leq |x-y| \leq 1+\epsilon \} \lesssim \epsilon, \end{equation} where $\mu$ is a Borel measure satisfying the energy estimate $I_s(\mu)=\int \int {|x-y|}^{-s} d\mu(x) d\mu(y)<\infty$ for $s>\frac{d+1}{2}$. An example due to Mattila (\cite{Mat87}, remark 4.5; \cite{Mat85}) shows in two dimensions that for no $s<\frac{3}{2}$ does $I_s(\mu)<\infty$ imply (\ref{key}). His construction can be extended to three dimensions. Mattila's example readily applies to the case when the Euclidean norm in (\ref{key}) is replaced by a norm generated by a convex body with a smooth boundary and non-vanishing Gaussian curvature. 

In this paper we prove, for all $d \ge 2$, that for no $s<\frac{d+1}{2}$ does $I_s(\mu)<\infty$ imply (\ref{key}) or the analogous estimate where the Euclidean norm is replaced by the norm generated by a particular convex body $B$ with a smooth boundary and everywhere non-vanishing curvature. Our construction, based on a combinatorial construction due to Pavel Valtr (\cite{V05}) naturally leads us to some interesting connections between the problem under consideration, geometric incidence theorem in the discrete setting and distribution of lattice points in convex domains. 

We also prove that Mattila's example can be discretized to produce a set of points and annuli for which the number of incidences is much greater than in the case of the lattice. In particular, we use the known results on the Gauss Circle Problem and a discretized version of Mattila's example to produce a non-lattice set of points and annuli where the number of incidences is much greater than in the case of the standard lattice. 

Finally, we extend Valtr's example into the setting of vector spaces over finite fields and show that a finite field analog of (\ref{key}) is also sharp. 
\end{abstract} 

\maketitle

%\tableofcontents

\section{Introduction}

The classical Falconer distance conjecture (\cite{Fal86}) says that if the Hausdorff dimension of a compact set $E \subset {\Bbb R}^d$, $d \ge 2$, is greater than $\frac{d}{2}$, then the Lebesgue measure of 
$$ \Delta(E)=\{|x-y|: x,y \in E\}$$ is positive. Here $|\cdot|$ denotes the Euclidean distance. The problem was introduced by Falconer in \cite{Fal86} where he proves that the Lebesgue measure of $\Delta(E)$, denoted by ${\mathcal L}^1(\Delta(E))$, is indeed positive if the Hausdorff dimension of $E$, denoted by $dim_{{\mathcal H}}(E)$ is greater than $\frac{d+1}{2}$. Since then, due to efforts of Bourgain (\cite{B94}), Erdogan (\cite{Erd05}), Mattila (\cite{Mat87}, \cite{M95}), Wolff (\cite{W99}) and others, the exponent has been improved, with the best current result due to Wolff in two dimensions (\cite{W99}) and Erdogan in higher dimensions (\cite{Erd05}). They proved that ${\mathcal L}^1(\Delta(E))>0$ provided that $dim_{{\mathcal H}}(E)>\frac{d}{2}+\frac{1}{3}$. See also \cite{MS99} where the authors prove that if $dim_{{\mathcal H}}(E)>\frac{d+1}{2}$, then $\Delta(E)$ contains an interval. 

Falconer's $\frac{d+1}{2}$ exponent follows from the following key estimate. Suppose that $\mu$ is Borel measure on $E$ such that 
$$ I_s(\mu)=\int \int {|x-y|}^{-s} d\mu(x) d\mu(y)<\infty$$ for every $s>\frac{d+1}{2}$. Then 
\begin{equation} \label{falconerestimate} \mu \times \mu \{(x,y): 1 \leq |x-y| \leq 1+\epsilon \} \lesssim \epsilon, \end{equation} where here and throughout, $X \lesssim Y$ means that there exists a uniform $C>0$ such that $X \leq CY$. 

This estimate follows by Plancherel and the fact that if $\sigma$ denotes the Lebesgue measure on the unit sphere, then 
\begin{equation} \label{stationaryphase} |\widehat{\sigma}(\xi)| \lesssim {|\xi|}^{-\frac{d-1}{2}}. \end{equation}

This implies, in particular, that (\ref{falconerestimate}) still holds if the Euclidean distance $|\cdot|$ is replaced by ${||\cdot||}_B$, where $B$ is a symmetric convex body with a smooth boundary and everywhere non-vanishing Gaussian curvature. This is because the estimate (\ref{stationaryphase}) still holds if $\sigma$ is replaced by $\sigma_B$, the Lebesgue measure on $\partial B$. To be precise, under these assumptions on $B$, the estimate 
\begin{equation} \label{Bfalconerestimate} \mu \times \mu \{(x,y): 1 \leq {||x-y||}_B \leq 1+\epsilon \} \lesssim \epsilon \end{equation} holds provided that $I_s(\mu)<\infty$ with $s>\frac{d+1}{2}$. 

A consequence of this more general version of (\ref{falconerestimate}) is that ${\mathcal L}^1(\Delta_B(E))>0$ whenever $dim_{{\mathcal H}}(E)>\frac{d+1}{2}$, where 
$$ \Delta_B(E)=\{{||x-y||}_B: x,y \in E\}.$$ 

\vskip.125in 

See, for example, \cite{AI04}, \cite{Hi05}, \cite{IL05}, \cite{IR07} and \cite{IR09} for the description of this generalization of the Falconer distance problem and its connections with other interesting problems in geometric measure theory and other areas. 

An example due to Mattila (see \cite{Mat85}) shows in two dimensions that for no $s<\frac{3}{2}$ does 
$$I_s(\mu)=\int \int {|x-y|}^{-s} d\mu(x)d\mu(y)<\infty$$ imply (\ref{Bfalconerestimate}). Mattila's construction can be generalized to three dimensions. However, in dimensions four and higher, his method does not seem to apply. It is important to note that in any
dimensions, an example due to Falconer (\cite{Fal86}) shows that for no $s<\frac{d}{2}$ does $I_s(\mu)<\infty$ imply that the estimate (\ref{Bfalconerestimate}) hold. We record these calculations for the reader's convenience in the Section \ref{mattila} below. 

In this paper we construct a measure in all dimensions which shows that for no $s<\frac{d+1}{2}$ does $I_s(\mu)<\infty$ imply that (\ref{Bfalconerestimate}) hold. More precisely, we have the following result. 
\begin{theorem} \label{main} There exists a symmetric convex body $B$ with a smooth boundary and non-vanishing Gaussian curvature, such that for any $s<\frac{d+1}{2}$, there exists a Borel measure 
$\mu_s$, such that $I_s(\mu) \approx 1$ and 
\begin{equation} \label{fail} \limsup_{\epsilon \to 0} \ \epsilon^{-1} \mu_s \times \mu_s \{(x,y): 1 \leq {||x-y||}_B \leq 1+\epsilon \}=\infty. \end{equation} 
\end{theorem}

\begin{remark} The proof will show that $\epsilon^{-1}$ in (\ref{fail}) may be replaced by $\epsilon^{-\frac{2s}{d+1}-\gamma}$ for any $\gamma>0$. We also note that we only need to establish (\ref{fail}) with $s \ge \frac{d}{2}$ since if $s<\frac{d}{2}$, the example due to Falconer (\cite{Fal86}), mentioned above, does the job. 

Another way of stating the conclusion of Theorem \ref{main} is that for no $s<\frac{d+1}{2}$ does $I_s(\mu)<\infty$ imply that the distance measure is in $L^{\infty}({\Bbb R})$. The distance measure $\nu$ is defined by the relation 
$$ \int g(t) d\nu(t)=\int \int g({||x-y||}_B) d\mu(x) d\mu(y).$$ 
\end{remark} 

\vskip.125in 

\subsection{Structure of the paper} Theorem \ref{main} is proved in Section \ref{proof} below. The idea is to make a construction for a specific convex body obtained by glueing the upper and lower hemispheres of the paraboloid. In the Subsection \ref{combinatorics} we describe the combinatorial construction used in the proof of Theorem \ref{main}. We also describe another construction, due to Lenz and point out that it cannot be used in our setting. The proof of this assertion is postponed to Section \ref{discretevscontinuous}. In Subsection \ref{assembly} we use the combinatorial construction from Section \ref{combinatorics} to complete the proof of Theorem \ref{main}. In Section \ref{mattila} we describe an example due to Mattila and generalize it three dimensions. In Section \ref{discretevscontinuous} we explain the connection between the estimates (\ref{falconerestimate}) and (\ref{Bfalconerestimate}) and how this connection rules out the Lenz construction. In Section \ref{latticepoints} we discuss the connection between Mattila's example in Section \ref{mattila} and distribution of lattice points in large disks. In the final section of the paper we extend Valtr's example into the setting of vector spaces over finite fields and show that the finite field analog of (\ref{key}), proved by Misha Rudnev and the first listed author in (\cite{IR07ii}) is also sharp. 

\vskip.125in 

\section{Proof of the main result} 
\label{proof} 

\vskip.125in 

\subsection{Combinatorial underpinnings} \label{combinatorics} The proof of Theorem \ref{main} uses a generalization of the two-dimensional construction due to Pavel Valtr (see \cite{BMP00}, \cite{V05}). A similar construction can also be found in \cite{K77} in a slightly different context. Let 
$$ P_n=\left\{\left(\frac{i_1}{n}, \frac{i_2}{n}, \dots, \frac{i_{d-1}}{n}, \frac{i_d}{n^2}\right):0\leq i_j\leq n-1,\text{ for } 1\leq j\leq d-1, \text{ and } 1\leq i_d \leq n^2\right\}.$$ 
Notice that in each of the first $d-1$ coordinates, there are $n$ evenly distributed points, but in the last dimension, there are $n^2$ evenly distributed points. Now, let 
$$ H=\{(t,t, \dots, t,t^2) \in {\Bbb R}^d: t \in {\Bbb R}\}$$ and define 
$$ L_H=\{H+p, p \in P_n\}.$$ 

\vskip.125in

\begin{figure}
\centering
\includegraphics[scale=1]{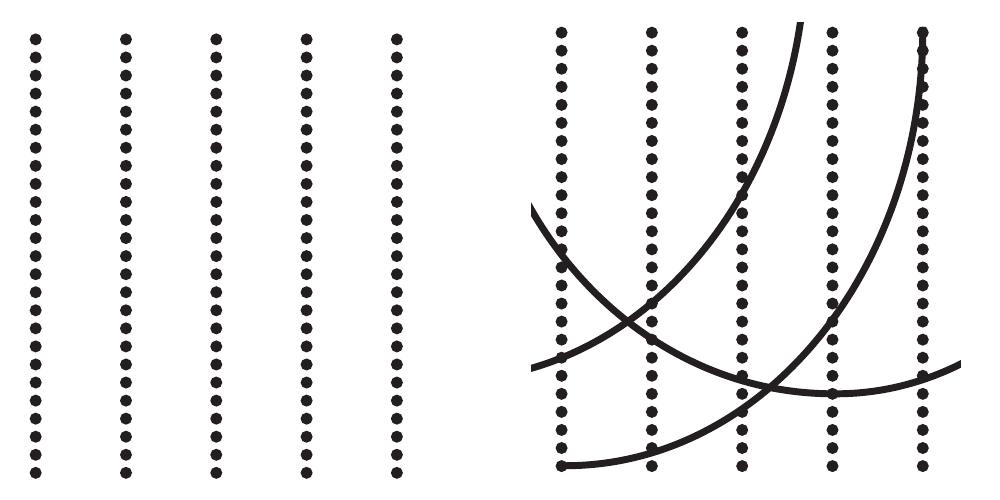}
\caption{On the left, we see a picture of the set $P_5$, on the right, we see it again with a few parabolic arcs, which intersect a point in each column.}
\label{ValtrFig1}
\end{figure}

Let $N=n^{d+1}$. By construction, $\# P_n= \# L_H = N$. Also by construction, each element of $L_H$ is incident to about $n^{d-1} \approx N^{\frac{d-1}{d+1}}$ elements of $P$. Thus the total number of incidences between $P$ and $L_H$ is 
$$\approx N^{1+\frac{d-1}{d+1}}=N^{\frac{2d}{d+1}}=N^{2-\frac{2}{d+1}}.$$

\subsubsection{Construction of the norm} \label{normconstruction} With this construction in hand, it is easy enough to flip the paraboloid upside down and glue it to another copy. Explicitly, let
$$B_U=\left\lbrace(x_1, x_2, \dots, x_d) \in \mathbb{R}^d : x_i \in [-1,1], \text{ for } 1 \leq i \leq d-1, \text{ and } x_d = 1-\left( x_1^2+x_2^2+ \dots + x_{d-1}^2 \right) \right\rbrace,$$
and
$$B_L=\left\lbrace(x_1, x_2, \dots, x_d) \in \mathbb{R}^d : x_i \in [-1,1], \text{ for } 1 \leq i \leq d-1, \text{ and } x_d = -1+ x_1^2+x_2^2+ \dots + x_{d-1}^2 \right\rbrace.$$
Now, let
$$B'=\left(B_U \cap \left\lbrace (x_1, x_2, \dots, x_d) \in \mathbb{R}^d :x_d \geq 0\right\rbrace \right) \cup \left( B_L \cap \left\lbrace (x_1, x_2, \dots, x_d) \in \mathbb{R}^d :x_d \leq 0 \right\rbrace \right).$$
Finally, define $B$ to be the convex body $B'$, with the ridge at the transition between $B_U$ and $B_L$ smoothed.

\vskip.125in 

Let $L$ denote $N$ copies of $\partial B$, each translated by an element of $P_n$.
Now we have a symmetric convex body $B \subset {\Bbb R}^d$ with a smooth boundary and everywhere non-vanishing curvature, a point set $P_n$ of size $N$ and a set $L$ of translates of $\partial B$, of size $\approx N$, such that the number of incidences between $P_n$ and $L$ is 
$\approx N^{2-\frac{2}{d+1}}$. 

The reader may be aware of the fact that in dimensions four and higher, a more dramatic combinatorial example is available. 

\vskip.125in 

\subsubsection{Lenz construction} (see e.g. \cite{BMP00}) \label{lenz} More precisely, choose $N/2$ points evenly spaced on the circle 
$$ \{(\cos(\theta), \sin(\theta), 0, 0): \theta \in [0, 2\pi)\}$$ and $N/2$ points evenly spaced on the circle 
$$ \{(0,0, \cos(\phi), \sin(\phi)): \phi \in [0,2\pi)\}.$$ Let $K_N$ be the union of the two point sets. It is not hard to check that all the distances between the points on one circle and the points on the other circle are equal to $\sqrt{2}$. It follows that the number of incidences between the points of $K_N$ and the circles of radius $\sqrt{2}$ centered at the points of $K_N$ is $\approx N^2$, which is about as bad as it can be and much larger than the $N^{2-\frac{2}{d+1}}$ obtained in the generalization of Valtr's example above. However, this construction will not help in the continuous setting due to certain peculiarities of the Hausdorff dimension. This, in turn, leads to some interesting combinatorial questions. We shall discuss this issue in the final section of this paper. 

\vskip.125in 

\subsection{Using combinatorial information to construct the needed measures} 
\label{assembly} 
Let $\frac{d}{2} \leq s<\frac{d+1}{2}$. There is no point going below $\frac{d}{2}$ because the lattice-based construction in \cite{Fal86} shows that (\ref{falconerestimate}) cannot hold in that regime. Partition ${[0,1]}^d$ into lattice cubes of side-lengths $\epsilon$ where $\epsilon^{-s}=N$ for some large integer $N$. Let 
$n=N^{\frac{1}{d+1}}$. Put $P_n$ in the unit cube and select any lattice cube which contains a point of $P_n$. Let $Q_n$ denote the set of centers of the selected lattice cubes.

Now, we define $L_\epsilon$ to be the union of the $\epsilon$-neighborhoods of the elements of $L$. That is, for every translate of $\partial B$ by an element, $p \in P_n$, let $l_p$ denote the locus of points that are within $\epsilon$ of the translate of $\partial B$ by $p$. Then $$L_\epsilon = \bigcup_{p\in P_n} l_p$$.

\begin{figure}
\centering
\includegraphics[scale=1]{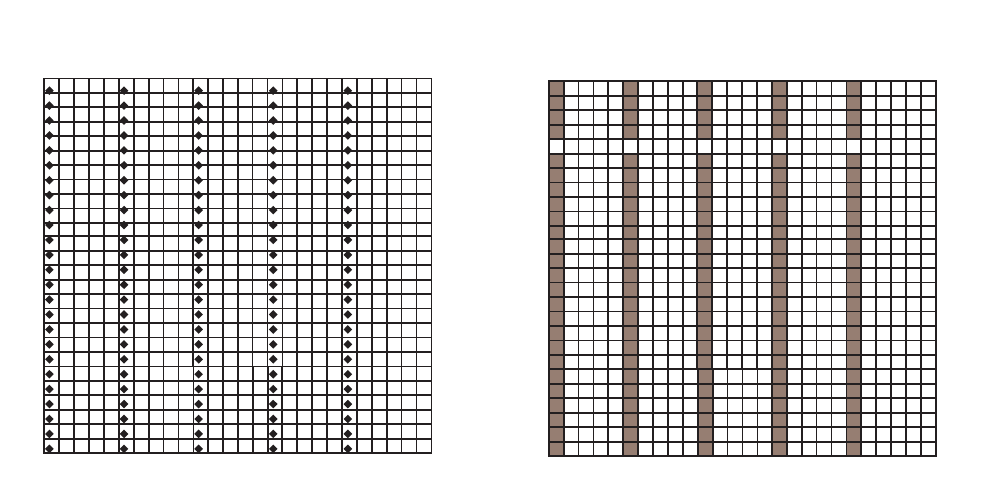}
\caption{On the left, we have $P_5$, in the partitioned unit cube. On the right, we filled in every cube which contained a point. Notice that there are gaps in the columns corresponding to cubes which did not contain any points.}
\label{ValtrFig2}
\end{figure}

\vskip.125in

\begin{lemma} \label{energy} Let $\mu_s$ denote the Lebesgue measure on the union of the colored cubes above, normalized so that 
$$ \int d\mu_s(x)=1.$$ 

More precisely, 
$$ d\mu_s(x)=\epsilon^{s-d} \sum_{p \in P_n} \chi_{R_{\epsilon}(p)}(x) dx,$$ where $R_{\epsilon}(p)$ denotes the cube of side-length $\epsilon$ centered at $p$. Then
$$ I_s(\mu_s) \approx 1.$$ 
\end{lemma} 

\vskip.125in 
\begin{proof}
To prove the lemma, observe that 

\vskip.125in 

$$ I_s(\mu) = \int \int {|x-y|}^{-s} d\mu_s(x) d\mu_s(y)$$
$$=\epsilon^{2(s-d)} \sum_{p,q \in P_n} \int \int {|x-y|}^{-s}  \chi_{R_{\epsilon}(p)}(x)  
\chi_{R_{\epsilon}(q)}(y) dxdy$$
$$=\epsilon^{2(s-d)} \sum_{p \in P_n} \int \int {|x-y|}^{-s}  \chi_{R_{\epsilon}(p)}(x)  
\chi_{R_{\epsilon}(p)}(y) dxdy$$
$$+\epsilon^{2(s-d)} \sum_{p \not=q \in P_n} \int \int {|x-y|}^{-s}  \chi_{R_{\epsilon}(p)}(x) \chi_{R_{\epsilon}(q)}(y) dxdy=I+II.$$

\vskip.125in 

We have 
$$ I=\epsilon^{2(s-d)} \sum_{p \in P_n} \int_{R_{\epsilon}(p)} \int_{R_{\epsilon}(p)} {|x-y|}^{-s} dxdy.$$

\vskip.125in 

Making the change of variables $X=x-y, Y=y$, we see that 

\vskip.125in 

$$ I \lesssim \epsilon^{2(s-d)} \sum_{p \in P_n} \epsilon^d \int_{|X| \leq \sqrt{d}\epsilon} {|X|}^{-s} dX$$
$$ \lesssim \epsilon^{2(s-d)} \cdot \epsilon^d \cdot \epsilon^{d-s} \sum_{p \in P_N} 1$$
$$ \lesssim \epsilon^s \cdot N \lesssim 1.$$ 

\vskip.125in 

On the other hand, 
$$ II \approx \sum_{p \not=q \in P_n} {|p-q|}^{-s} \epsilon^{2s}$$
$$=N^{-2} \sum_{p \not=q \in P_n} {|p-q|}^{-s}.$$

\vskip.125in 

We have 
$$p=(p',p_d)=\left( \frac{i_1}{n}, \dots, \frac{i_{d-1}}{n}, \frac{i_d}{n^2} \right)$$ and 
$$q=(q',q_d)=\left( \frac{j_1}{n}, \dots, \frac{j_{d-1}}{n}, \frac{j_d}{n^2} \right).$$ 

Let $i'=(i_1, \dots, i_{d-1})$ and $j'=(j_1, \dots, j_{d-1})$. 

\vskip.125in 

Thus we must consider 
$$ N^{-2} \sum_{i \not=j; |i'|, |j'| \leq n; i_d, j_d \leq n^2} {\left| \left| \frac{i'-j'}{n} \right|+\left| \frac{i_d-j_d}{n^2} \right| \right|}^{-s}.$$

\vskip.125in 

Replacing the sum by the integral, we obtain 

\vskip.125in 

$$ N^{-2} \int \int \dots \int_{\substack{ |i_1|, |j_1|, \dots, |i_{d-1}|, |j_{d-1}| \leq n\\ i_d, j_d \leq n^2}} {\left| \left| \frac{i'-j'}{n} \right|+\left| \frac{i_d-j_d}{n^2} \right| \right|}^{-s}di'dj'di_ddj_d,$$

\vskip.125in 

which, by a change of variables, $u' = (i'/n), u_d = (i_d/n^2), v'=(j'/n),$ and $v_d = (j_d/n^2),$ with similarly named coordinates, becomes
$$=\int \int \dots \int_{\substack{u \not=v\\  |u_1|, |v_1|, \dots, |u_{d-1}|, |v_{d-1}| \leq 1\\ u_d, v_d \leq 1}} {|u-v|}^{-s} du'dv'du_ddv_d \lesssim 1.$$

\vskip.125in 

This completes the proof of Lemma \ref{energy}. 
\end{proof}
\vskip.125in 

We are now ready to complete the argument in the case of the paraboloid. We have 
$$ \mu_s \times \mu_s \{(x,y): 1 \leq {||x-y||}_B \leq 1+\epsilon \}$$ is $\approx C \epsilon^{2s}$ times the number of incidences between the elements of $Q_n$ and $L_{\epsilon}$, where $Q_n$ and $L_{\epsilon}$ are constructed in the beginning of this section. Invoking our generalization of Valtr's construction from Section \ref{combinatorics} above, we see that 
$$ \mu_s \times \mu_s \{(x,y): 1 \leq {||x-y||}_B \leq 1+\epsilon \} \approx \epsilon^{2s} \cdot 
N^{2-\frac{2}{d+1}} \approx N^{-\frac{2}{d+1}}.$$ 

This quantity is much greater than $\epsilon=N^{-\frac{1}{s}}$ when $s<\frac{d+1}{2}$. This completes the proof of Theorem \ref{main}. 

\vskip.125in 

\vskip.125in 

\section{Mattila's construction} 
\label{mattila} 

In this section we describe Mattila's construction from \cite{Mat85} and its generalization to three dimensions.

First, we review the method of constructing a Cantor set of a given Hausdorff dimension, (see \cite{M95}). If we want a Cantor set, $\mathcal{C}_\alpha$, of Hausdorff dimension $0<\alpha<1$,  we need to find the $0<\lambda<1/2$ which satisfies $\alpha = \log 2 / \log (1/\lambda).$ Start with the unit segment, then remove the interval $(\frac{1}{2}- \lambda /2, \frac{1}{2}+\lambda/2)$. Next, remove the middle $\lambda$-proportion of each of the remaining subintervals, and so on. The classic ``middle-thirds" Cantor set would be generated with $\lambda = 1/3$.

To construct the two-dimensional example, $\mathcal{M}_2(\alpha)$, we let $F=\left(\mathcal{C}_\alpha\right)\cup \left(\mathcal{C}_\alpha-1\right)$. Then define $\mathcal{M}_2(\alpha) = F \times [0,1].$ Define the measure $\mu$ to be $\left(\mathcal{H}^\alpha|F\right)\times\left(\mathcal{L}^1|[0,1]\right),$ where $\mathcal{H}^\alpha$ is the $\alpha$-dimensional Hausdorff measure.

Pick a point $x=(x_1,x_2) \in \mathcal{M}_2(\alpha)$. Notice that if $x_1\in F$, either $x_1+1$ or $x_1-1$ is also in $F$. So there is an $\epsilon$-annulus, with radius 1, centered at $x$, which contains a rectangle of width $\epsilon$ and length $\sqrt{\epsilon}$. This rectangle intersects $\mathcal{M}_2(\alpha)$ lengthwise. The measure of this intersection is $\epsilon^{1/2+\alpha}$. This follows easily from the fact that the circle has non-vanishing curvature. It follows that 

$$ \mu\left\{y:1\leq|x-y|\leq1+\epsilon\right\} \gtrsim \epsilon^{\alpha+1/2} $$ for every $x$. It follows that 
$$ \mu \times \mu \{(x,y): 1 \leq |x-y| \leq 1+\epsilon \}$$
$$=\int  \mu\left\{y:1\leq|x-y|\leq1+\epsilon\right\} d\mu(x)$$
$$ \gtrsim \epsilon^{\alpha+1/2}.$$ 

We conclude that 
$$ \mu \times \mu \{(x,y): 1 \leq |x-y| \leq 1+\epsilon \} \lesssim \epsilon$$ only if 
$$ \epsilon^{\alpha+\frac{1}{2}} \lesssim 1,$$ which can only hold if 
$$ \alpha \ge \frac{1}{2}.$$ 

Thus the estimate (\ref{Bfalconerestimate}) does not in general hold for sets with Hausdorff dimension less than $\frac{3}{2}$. Letting $\alpha$ get arbitrarily small yields a family of counterexamples with Hausdorff dimensions arbitrarily close to $1$ , below which there are already counterexamples. See, for example, \cite{Fal86}. Note that we worked in $[-1,1] \times [0,1]$ instead of $[0,1] \times [0,1]$, to allow the main point to shine.

To construct $\mathcal{M}_3(\delta)$, the three-dimensional example, we set
$$\mathcal{M}_3(\delta) = \left(\mathcal{C}_\alpha \cup \mathcal{C}_\alpha-1\right) \times \left(\mathcal{C}_\alpha \cup \mathcal{C}_\alpha-1\right) \times \mathcal{C}_\beta,$$
where $\alpha = 1-\delta$, and $\beta = \delta/2$, and $\delta$ is determined later. We will set $\mu$ to be a product of the appropriate Hausdorff measures restricted to this set, much like the previous example. Notice that $\mathcal{M}_3(\delta)$ has a Hausdorff dimension of $2-\frac{3}{2}\delta$, and for a given point $x \in \mathcal{M}_3(\delta)$, there is an $\epsilon^{\frac{1}{2}}$ by $\epsilon^{\frac{1}{2}}$ by 
$\epsilon$ box inside the annulus whose measure is 
$$\epsilon^{\alpha/2} \cdot \epsilon^{\alpha/2} \cdot \epsilon^{\beta} = \epsilon^{1-\frac{\delta}{2}}.$$

Once again, we have used the fact that the sphere has non-vanishing Gaussian curvature, which implies, by elementary geometry, that the $\epsilon$-annulus contains and $\epsilon^{\frac{1}{2}}$ by $\epsilon^{\frac{1}{2}}$ by $\epsilon$ box. It follows that 

$$ \mu\left\{y:1\leq|x-y|\leq1+\epsilon\right\} \gtrsim \epsilon^{1-\frac{\delta}{2}}$$ for every $x$, which means that 
$$ \mu \times \mu \{(x,y): 1 \leq |x-y| \leq 1+\epsilon \} \gtrsim \epsilon^{1-\frac{\delta}{2}},$$ so (\ref{Bfalconerestimate}) does not hold. 

Thus we shown that for $s<2=\frac{d+1}{2}$ (when $d=3$), $I_s(\mu)<\infty$ does not imply that (\ref{Bfalconerestimate}) holds. Obvsere that both constructions in this sections work for any convex $B$ such that $\partial B$ is smooth and has everywhere non-vanishing Gaussian curvature. 

In Chapter \ref{latticepoints} below we shall see that Mattila's example can be used to construct a discrete point set and a family of annuli such that the number of incidences is much greater than it is in the corresponding problem where the point set is a lattice. 

It is interesting to note that Mattila's example can be adapted (see \cite{CEHIT10}) to study the sharpness of the following generalization of the Falconer estimate: 
$$ \mu \times \dots \times \mu \{(x^1, \dots, x^{k+1}): t_{ij} \leq {||x^i-x^j||}_B \leq t_{ij}+\epsilon \} \lesssim \epsilon^{k+1 \choose 2}.$$ 

This estimate is used to study $k$-simplexes in the way that the Falconer's estimate is adapted to study distance, which may be viewed as $1$-simplexes, or two-point configurations. This problem may be viewed as a fractal analog of the results initially studied by Furstenberg, Katznelson and Weiss (\cite{FKW90}). 

\vskip.125in 

\section{Connections between the inequality (\ref{Bfalconerestimate}) and the discrete incidence theorem} 
\label{discretevscontinuous}

\vskip.125in 

The reason we are able to use a generalization of Valtr's construction and are unable to use the Lenz construction described in subsubsection \ref{lenz} is due to Lemma \ref{energy}. Roughly speaking, the idea is the following. We want to take a discrete set of points in ${\Bbb R}^d$ and turn it into a set of Hausdorff dimension $s>0$ by thickening each of the $N$ points by $N^{-\frac{1}{s}}$. The resulting set has positive Lebesgue measure, of course, but the measure tends to $0$ as $N$ approaches infinity and we may view this set as $s$-dimensional if the energy integral 
$$ I_s(\mu)=\int \int {|x-y|}^{-s} d\mu(x) d\mu(y) \approx 1,$$ where $d\mu$ is the Lebesgue measure on the resulting union of $N^{-\frac{1}{s}}$ balls, normalized so that 
$$ \int d\mu(x)=1.$$ 

Unfortunately, this is not always possible as the Hausdorff dimension measures not just the amount of mass present, but also its distribution within the set. This issue was taken up in \cite{Hi05}, \cite{IL05} and \cite{IJL09} in the context of Delone (or homogeneous) sets and in \cite{IRU09} in more generality. 

More precisely, a discrete set can be turned into a fractal set of uniform Hausdorff dimension $s$ by this procedure if the following condition holds. 
\begin{definition} \label{adaptable} Let $P_N$ be a nested family of discrete subsets of the unit cube ${[0,1]}^d$, $d \ge 2$, consisting of $N$ points. We say that $P_N$ is $s$-adaptable if 

\vskip.125in 

$$ N^{-2} \sum_{p,q \in P_N} {|p-q|}^{-s} \lesssim 1$$ with constants independent of $N$. 
\end{definition} 

The $s$-adaptability condition holds for Valtr's construction that we generalize in this paper, but it is not true for the Lenz construction. We now make this claim precise. 

\begin{proposition}
The Lenz construction, $K_N$, is not $s$-adaptable for $s>1$.
\end{proposition}
\begin{proof}
Let $K_N'$ denote the points on the circle in the first quadrant of the first two dimensions.
\begin{align*}
N^{-2} \sum_{p,q \in K_N} {|p-q|}^{-s} &\gtrsim N^{-2} \sum_{p,q \in K_N'} {|p-q|}^{-s}\\
&\gtrsim N^{-2} \sum_{p \in K_N'} \sum_{\substack{q\in K_N'\\|p-q|\leq 1/2}} {|p-q|}^{-s}\\
&\gtrsim N^{-2} N \sum_{j=1}^{\frac{N}{6}} \left(\frac{j}{2N}\right)^{-s}\\
&\gtrsim N^{s-1} \sum_{j=1}^{\frac{N}{6}} j^{-s}\\
&\gtrsim N^{s-1},
\end{align*}
where, in the second inequality, when we fix a point $p$, we measure the distance to points $q$ which are close enough, so that in the next inequality, these distances are at least one-half of $\frac{j}{N}$, where the point $q$ is $j$ points away from $p$ along the circle. The quantity $N^{s-1}$ is unbounded as $N$ grows large for $s>1$.
\end{proof}

Notice that the proof did not rely on the ambient dimension at all. This raises the possibility of studying the Erd\H os single distance conjecture in higher dimensions under the assumption that the set is $s$-adaptable. This definition may be viewed as a natural generalization of the notion of homogeneous sets, used, for example, by Solymosi and Vu in \cite{SV04}. 

%One can adapt the techniques in \cite{IJL09}, originally applied to homogeneous sets, to prove the following result. 
%\begin{theorem} \label{incidencetheorem} Suppose that $P_N$ is $s$-adaptable with $s \ge \frac{d+1}{2}$. Then for any $t>0$, 
%\begin{equation} \label{incidence} \# \{(x,y) \in P_N: {||x-y||}_B=t \} \lesssim N^{2-\frac{1}{s}}.\end{equation} for any convex body $B$ with a smooth boundary and everywhere non-vanishing Gaussian curvature. 
%\end{theorem} 

%The question is whether one can push the range of exponents $s$ for which (\ref{incidence}) holds below $\frac{d+1}{2}$. Valtr's example shows that this is not possible for general metrics, even the ones generated by norms where the boundary of the underlying convex body is smooth and has non-vanishing Gaussian curvature. Mattila's example, described in Section \ref{mattila} above, shows that this is not possible in dimensions two and three for any norm generated by a convex body $B$ with a smooth boundary and everywhere non-vanishing Gaussian curvature. The question is, what happens in dimension four and higher. We hope to address this question in a subsequent paper. 

\vskip.125in 

\section{Connections with the distribution of lattice points in convex domains} 
\label{latticepoints} 

One of the main themes of this paper is to carefully examine the interplay between discrete and continuous theory. In this section, we present a natural question involving geometric incidences, whose discrete analog, the Erd\H os single distance problem, has already received a fair amount of attention. For the single distance problem, the integer lattice exhibits the most extreme behavior of which we are aware. In our related incidence problem, we notice that there are sets which have more extreme behavior than the lattice. In what follows, we will compute the number of incidences for the lattice, which will involve the Gauss Circle Problem. Then we will give lower bounds for the number of incidences of another set, which is analogous to $\mathcal{M}_2(\alpha)$ above.

The Erd\H os single distance problem asks, given a fixed number, $N$, how often a particular distance can be determined by pairs of points from any set of $N$ points. The conjecture is that no distance can occur more often than $\lessapprox N$ times. This bound is attained when the $N$ points are arranged in an $\sqrt{N} \times \sqrt{N}$ lattice. Currently, the best estimate is $N^{\frac{4}{3}}$, by Spencer, Szemer\'edi, and Trotter, in \cite{SST84}. See also, Sz\'ekely, \cite{Sz97}.

This can easily be seen as a question of incidences between points and circles of a fixed radius. We explore the same question of incidences, but with points and annuli of a fixed radius. Interestingly enough, if we properly convert the set $\mathcal{M}_2(\alpha)$, into a discrete point set, it can have more point-annulus incidences than the lattice.

After showing the two-dimensional cases in full detail, we show that there is a similar phenomenon in three dimensions by constructing a discrete analog of $\mathcal{M}_3(\delta)$, and comparing the number of incidences there with the number incidences in a three-dimensional lattice.

For each of the sets below, the parameter $s$ will play an important role. In these examples, it can be effectively viewed as a discrete analog of the Hausdorff dimension. This will help to illustrate some of the direct connections between continuous and discrete theory. These incidence examples will serve to demonstrate precisely where some of the connections break down. They will also provide more evidence that the circle is a special case, exemplifying the delicate balance between extremal behavior and expected behavior in geometric incidence problems.

\vskip.125in 

The Gauss Circle Problem asks, given a symmetric, convex body $B$, for the size of the discrepancy term as a function of the radius of the body, $R$.
$$ D(R)=\# \{RB \cap {\Bbb Z}^d \}-R^d|B| \ \text{as} \ R \to \infty.$$

When $B$ is the unit ball, it is known that 
$$ |D(R)| \lesssim R^{d-2}, \ \text{for} \ d \ge 5$$ and 
$$ |D(R)| \lessapprox R^2, \ \text{when} \ d=4,$$ 

In two and three dimensions, the problem is far from solved, with the best known estimates due to Heath-Brown (\cite{HB97}) and Huxley (\cite{Hux90}), respectively. See \cite{Hux96} and the references contained therein for the description of the problem and related results. 
For $d=2$, the Gauss Circle Conjecture states that $|D(R)| \approx R^{1/2} \log R$. 

\subsection{The two-dimensional case}
\subsubsection{The lattice example}

The result we need is the most recent bound by Huxley, namely

\begin{equation}\label{huxley}
|D(R)| \lesssim R^{131/208}\left(\log(R)\right)^{18627/8320}.
\end{equation}

Given $N$ points, construct a $\sqrt{N} \times \sqrt{N}$ lattice, and scale it down so that it fits into the unit square. The number of incidences between points in the lattice and annuli centered at points in the lattice will be about the number of points in an annulus centered at the origin times the number of points in the set. To count the number of points in the annulus of thickness $\epsilon = N^{-1/s}$ and radius proportional to say, one-tenth, we will count the number of points in the circle of radius $1/10 + \epsilon$ and subtract the number of points in the circle of radius $1/10$. Both of these counts will have discrepancies. In order to get an accurate bound on the number of points in each circle, we will appeal to \ref{huxley}.

Let $a$ be the number of points in the annulus centered at the origin, with radius $1/10$ and thickness $\epsilon$. Let $c(R)$ denote the number of lattice points in the ball of radius $R$, centered at the origin, where the appropriate dimension is assumed. We will compute the number of integer lattice points in the annulus of radius $\sqrt{N}/10$, with thickness $\epsilon \sqrt{N}$, centered at the origin, as this will be identically equal to $a$, as it is the same setting as before, but scaled up  by $\sqrt{N}$. Then,

\begin{align*}
a &\approx c\left(\frac{\sqrt{N}}{10} + \sqrt{N} \cdot N^{-\frac{1}{s}}\right)- c\left(\frac{\sqrt{N}}{10}\right)\\
&= \pi\left(\frac{\sqrt{N}}{10} + \sqrt{N} \cdot N^{-\frac{1}{s}}\right)^2 - \left|D\left(\frac{\sqrt{N}}{10} + \sqrt{N} \cdot N^{-\frac{1}{s} }\right)\right| - \pi\left(\frac{\sqrt{N}}{10} \right)^2 + \left|D\left(\frac{\sqrt{N}}{10} \right)\right|.\\
&\approx N^{\frac{1}{2}-\frac{1}{s}} + \left|D\left(\sqrt{N}\right)\right|.
\end{align*}

We want the first term to dominate. Notice that by (\ref{huxley}), we can bound the error term by

\begin{align*}
\left|D\left(\sqrt{N}\right)\right| &\lesssim \left(\sqrt{N}\right)^{\frac{131}{208}}\left(\log\left(\sqrt{N}\right)\right)^{\frac{18627}{8320}}\\
&\lesssim \left(\sqrt{N}\right)^{\frac{131}{208}}\log N.
\end{align*}

So we need to be sure that

$$ N^{\frac{1}{2}-\frac{1}{s}} \gtrsim \left(\sqrt{N}\right)^{\frac{131}{208}}\log N,$$
which happens when $s > \frac{416}{285}$. Note that this reduces to $\frac{4}{3}$ if we assume the full conjecture.

Now, in our range on $s$, we have that the number of incidences between points in our set and annuli is $I_1=Na \approx N^{2-\frac{1}{s}}$.

\subsubsection{The Mattila-type example}

To construct this example, we construct a discrete analog of the set $\mathcal{M}_2(\alpha)$. For a given $N$ and $\alpha$, the set will consist of $N$ points. Let $M =N^{\frac{1}{1+\alpha}}$, and $m = 2^{\frac{1}{\alpha}},$ so that $\alpha = \log 2/ \log m.$ Also, let $s=1+\alpha$.

Start with the unit segment and delete middle intervals of proportion $1/m$, to form a Cantor-like set, but only perform $n$ deletions, where $n$ satisfies $m^n=M$. Now, replace each interval by a single point located at the center of the interval. Call this set $\mathcal{C}_{\alpha,n}$. There should now be $2^n$, or $M^{\alpha}$ points, each located at the center of what was an interval of a Cantor-like set. Now, define the set $F_n$ to be $\mathcal{C}_{\alpha,n} \cup \left(\mathcal{C}_{\alpha,n}-1\right)$.

Now, define the set $\mathcal{L}_n$ to be $M$ points, spaced evenly in the interval $[0,1]$. Our new set will be called $\mathcal{M}_{2,n}(\alpha)$, and it will be formed by taking the Cartesian product of $F_n$ and $\mathcal{L}_n.$ Notice that it consists of $N$ points.

\begin{figure}
\centering
\includegraphics[scale=1]{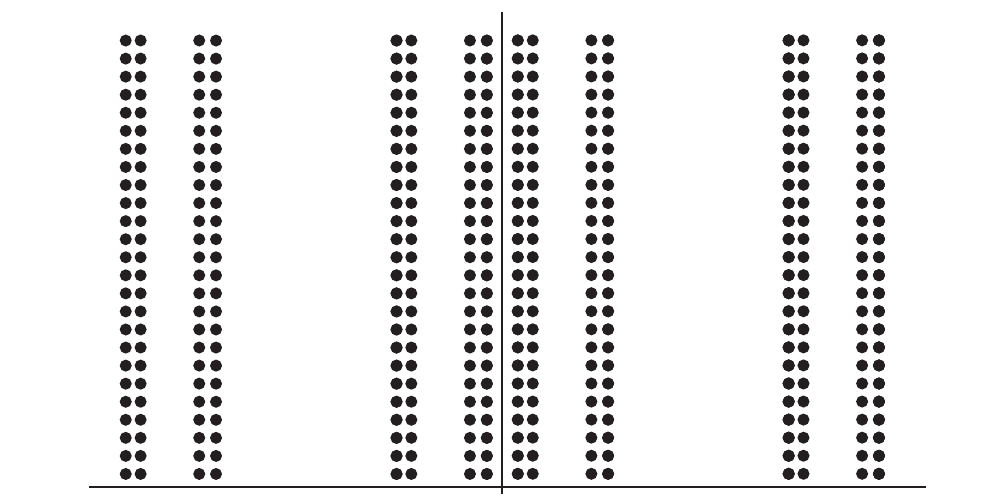}
\caption{Here we see $2^n$ points located at the centers of the intervals of a Cantor-like set in the $x$ direction, and $m^n$ points evenly spaced in the $y$ direction.}
\label{MattilaFigure}
\end{figure}

Now, as before, in any annulus of radius 1, there will be a rectangle of length $m^{-n/2}$, and width $m^{-n}$, which contains about $m^{n/2}$ points. Since there are $N$ such annuli, the total number of incidences will be about
$$I_2=Nm^{n/2} = M^{\frac{3}{2} + \alpha} = N^\frac{\frac{3}{2}+\alpha}{1+\alpha} = N^{1+\frac{1}{2s}}.$$

\subsubsection{Comparing the lattice to the Mattila-type set}

The final step is to actually compare the two sets. Notice that we have allowed each of them to be determined completely by two parameters, $N$ and $s$. We will show that for a particular range of $s$, the Mattila-type set, $\mathcal{M}_{2,n}(\alpha)$, has more incidences than the lattice.

\begin{align*}
I_1 &\lesssim I_2\\
N^{2-\frac{1}{s}} &\lesssim N^{1+\frac{1}{2s}}\\
1 &< \frac{3}{2s}\\
s &< \frac{3}{2}.\\ 
\end{align*}
 
Recalling the ranges of $s$ for which both incidence estimates hold, we note that these estimates are only valid when $s>\frac{416}{285}$, although this would be reduced to $s>\frac{4}{3}$ if we assumed the full strength of the Gauss Circle Conjecture. In summary, if $s$ is within the prescribed range, $\mathcal{M}_{2,n}(\alpha)$ provides more point-annulus incidences than the lattice.

\subsection{The three-dimensional case}

\vskip.125in 

\subsubsection{The lattice example}

The three-dimensional lattice is a Cartesian product of three copies of the set of $N^{\frac{1}{3}}$ evenly spaced points in the unit segment. Of course, it has $N$ points total. In three dimensions, instead of annuli, we will use spherical shells of radius $\frac{1}{10}$, and thickness $N^{-\frac{1}{s}}$. Let $b$ be the number of lattice points in the shell centered at the origin. As before, we can scale up so that the points are 1-seperated, then appeal to the most recent results of the Gauss Circle Problem. In three dimensions, we use the exponent in \cite{HB97}, which bounds the discrepancy by
\begin{equation}\label{HB}
|D(R)| \lesssim R^{\frac{21}{16}}.
\end{equation}

Let $b$ denote the number of lattice points in the spherical shell centered at the orign.

\begin{align*}
b &= c\left(\left(\frac{N}{10}\right)^{\frac{1}{3}} + N^{\frac{1}{3}-\frac{1}{s}}\right)- c\left(\left(\frac{N}{10} \right)^{\frac{1}{3}}\right)\\
&= \pi\left(\left(\frac{N}{10} \right)^{\frac{1}{3}} + N^{\frac{1}{3}-\frac{1}{s}}\right)^3 - \left|D\left(\left(\frac{N}{10} \right)^{\frac{1}{3}} + N^{\frac{1}{3}-\frac{1}{s}}\right)\right| - \pi\left(\left(\frac{N}{10} \right)^{\frac{1}{3}}\right)^3 + \left|D\left(\left(\frac{N}{10} \right)^{\frac{1}{3}}\right)\right|\\
&\approx N^{1-\frac{1}{s}} + \left|D\left(N^{\frac{1}{3}}\right)\right|.
\end{align*}

Again, when estimate the discrepancy using Heath-Brown's bound, (\ref{HB}), we get that our estimate on the number of points in the spherical shell is valid only when it is much, much larger than the error term. To be specific,

$$ N^{\frac{1}{3}-\frac{1}{s}} \gtrsim \left(\sqrt{N}\right)^{\frac{21}{16}}, $$
which happens when $s > \frac{16}{9}$. Note that this reduces to $\frac{3}{2}$ if we assume the full conjecture.

Now, in our range on $s$, we have that the number of incidences between points in our set and annuli is $I_3=Nb \approx N^{2-\frac{1}{s}}$.

\vskip.125in 

\subsubsection{The Mattila-type example}

We will construct the discrete analog of $\mathcal{M}_{3,n}(\delta)$ in a very similar manner to the way we constructed $\mathcal{M}_{2,n}(\alpha)$. Given a large integer, $N$, we let $\delta$ be chosen later. Set $\alpha = 1-\delta$, and $\beta=\frac{\delta}{2}$. Our parameter, $s$, will be $2\alpha + \beta$, or $2- \frac{3 \delta}{2}.$ Let $M$ be such that $N=M^s$. Now, let $m_1$ satisfy 
$$\alpha = \frac{\log2}{\log m_1},$$ and let $m_2$ satisfy 
$$\beta = \frac{\log 2}{\log m_2}.$$ 

\vskip.125in 

The biggest difference in construction between this point set and $\mathcal{M}_{2,n}(\alpha)$ is that there will be no ``Lebesgue" dimension. We will actually take the Cartesian product of three Cantor-like sets, which are to be constructed as above. Specifically, we will have
$$\mathcal{M}_{3,n}(\delta) = \mathcal{C}_{\alpha,n} \times \mathcal{C}_{\alpha,n} \times \mathcal{C}_{\beta,n}.$$

Yet again, in any annulus of radius 1, there will be a rectangle of length $m_1^{-\frac{n}{2}}$, width $m_1^{-\frac{n}{2}}$, and depth $m_2^{-n}$, which contains about $m^{n}$ points. Since there are $N$ such annuli, the total number of incidences will be about
$$I_4=Nm^{n} = N^{1+\frac{\alpha}{2\alpha+\beta}} = N^{1+\frac{1-\delta}{2-3\frac{\delta}{2}}}=N^{\frac{5}{3}-\frac{1}{3s}}.$$

\vskip.125in 

\subsubsection{Comparing the lattice to the Mattila-type set}

Again, we notice that for a particular range on $s$, $\mathcal{M}_{3,n}(\delta)$ has more incidences than the lattice.

\begin{align*}
I_3 &\lesssim I_4\\
N^{2-\frac{1}{s}} &\lesssim N^{\frac{5}{3}+\frac{1}{3s}}\\
1/3 &< \frac{2}{3s}\\
s &< 2.\\ 
\end{align*}

Recalling, as before, the ranges of $s$ for which both incidence estimates hold, we note that these estimates are only valid when $s>16/9$, although this would be reduced to $s>3/2$ if we assumed the full strength of the conjecture. In summary, if $s$ is within the prescribed range, $\mathcal{M}_{3,n}(\delta)$ provides more incidences between points and spherical shells than the lattice.

\vskip.125in 

\section{An analog of Falconer's estimate in vector spaces over finite fields} 

\vskip.125in 

In this section we shall see that the phenomenon captured by Valtr's example still persists in vector spaces over finite fields. Let $E \subset {\Bbb F}_q^d$, the $d$-dimensional vector space over the field with $q$ elements. The following result from \cite{IR07ii} may be viewed as an analog of (\ref{falconerestimate}). 
\begin{theorem} Let  $E \subset {\Bbb F}_q^d$, $d \ge 2$, and let $|E|$ denote its size. Let 
$$ S_t=\{x \in {\Bbb F}_q^d: x_1^2+x_2^2+\dots+x_d^2=t \}.$$ 

Then 
$$ |\{(x,y) \in E \times E: x-y \in S_t \}|=\frac{{|E|}^2}{q}+{\mathcal D}(q),$$ as long as $t \not=0$, where 
$$ |{\mathcal D}(q)| \leq 2|E|q^{\frac{d-1}{2}}.$$ 
\end{theorem} 

It follows that if $|E|$ is much larger than $q^{\frac{d+1}{2}}$, then 
\begin{equation} \label{finitefalconerestimate} |\{(x,y) \in E \times E: x-y \in S_t \}|=\frac{{|E|}^2}{q}(1+o(1)). \end{equation} 

Our first observation, already implicit in (\cite{IR07ii}), is that the same estimate holds if $S_t$ is replaced by any set $\Gamma \subset {\Bbb F}_q^d$ such that 
$$|\Gamma|=q^{d-1}(1+o(1))$$ and 
\begin{equation} \label{salem} |\widehat{\Gamma}(m)| \lesssim q^{-\frac{d+1}{2}} \end{equation} for any $m \not=(0, \dots, 0)$. Here the Fourier transform of a function $f: {\Bbb F}_q^d \to {\mathcal C}$ is defined by the relation 
$$ \widehat{f}(m)=q^{-d} \sum_{x \in {\Bbb F}_q^d} \chi(-x \cdot m) f(x),$$ where $\chi$ is the principal character on ${\Bbb F}_q$. 

The fact that (\ref{salem}) holds for $\Gamma=S_t$, $t \not=0$, follows from a rather deep fact, due to Weil and Salie, that the twisted Kloosterman sum satisfies the bound 
$$ \left| \sum_{s \not=0} \chi \left(as+\frac{1}{s} \right) \gamma(s) \right| \leq 2\sqrt{q},$$ where $\gamma$ is a multiplicative character of order $2$. See \cite{Sa32} and\cite{We48}. See also \cite{IR07ii} and the references contained therein. 

\begin{theorem} \label{general} Let $E \subset {\Bbb F}_q^d$, $d \ge 2$. Suppose that $\Gamma \subset {\Bbb F}_q^d$ satisfies (\ref{salem}). Then 
$$ |\{(x,y) \in E \times E: x-y \in \Gamma \}|=\frac{{|E|}^2}{q}+{\mathcal D}(q),$$ where 
$$ |{\mathcal D}(q)| \lesssim |E|q^{\frac{d-1}{2}}.$$ 
\end{theorem} 

In order to establish Theorem \ref{general}, observe that  
$$ |\{(x,y) \in E \times E: x-y \in \Gamma \}|$$
$$=\sum E(x)E(y)\Gamma(x-y)$$
$$=\frac{{|E|}^2}{q}+q^{2d} \sum_{m \not=(0, \dots, 0)} {|\widehat{E}(m)|}^2 \widehat{\Gamma}(m)$$
$$=\frac{{|E|}^2}{q}+{\mathcal D}(q),$$ where by (\ref{salem}), 
$$ |{\mathcal D}(q)| \lesssim q^{2d} \cdot q^{-\frac{d+1}{2}} \cdot \sum_m {|\widehat{E}(m)|}^2$$ 
$$ =q^d \cdot q^{\frac{d-1}{2}} \cdot q^{-d} \sum_x {|E(x)|}^2=|E|q^{\frac{d-1}{2}}.$$ 

In the manipulations above we used the size condition on $\Gamma$, (\ref{salem}) and the facts that 
$$ f(x)=\sum_{m \in {\Bbb F}_q^d} \chi(x \cdot m) \widehat{f}(m)$$ and 
$$ \sum_{m \in {\Bbb F}_q^d} {|\widehat{f}(m)|}^2=q^{-d} \sum_{x \in {\Bbb F}_q^d} {|f(x)|}^2.$$ 

It follows that if $\Gamma \subset {\Bbb F}_q^d$, $d \ge 2$, with $|\Gamma|=q^{d-1}(1+o(1))$ and (\ref{salem}) holds, then 
\begin{equation} \label{finiteBfalconerestimate} |\{(x,y) \in E \times E: x-y \in \Gamma \}|=\frac{{|E|}^2}{q}(1+o(1)) \end{equation} if $|E|$ is much greater than $q^{\frac{d+1}{2}}$. 

\vskip.125in 

\subsection{Sharpness of the finite field version of Falconer (estimates (\ref{finitefalconerestimate}) and (\ref{finiteBfalconerestimate}))} 

\subsubsection{The spherical case} Theorem 2.7 in \cite{HIKR10} implies that if $d$ is odd, the (\ref{finitefalconerestimate}) does not hold. More precisely, the proof contains a construction of a set $E$ such that 
\begin{equation} \label{rudnev} |E|=\frac{1}{2}cq^{\frac{d+1}{2}} \ \text{and} \ |\Delta(E)| \leq cq,
\end{equation} where 
$$ \Delta(E)=\{||x-y||: x,y \in E\}$$ with 
$$ ||x||=x_1^2+x_2^2+\dots+x_d^2.$$ 

Keeping (\ref{rudnev}) in mind, observe that if 
$$ |\{(x,y) \in E \times E: x-y \in S_t \}| \leq C \frac{{|E|}^2}{q}$$ for all $t \not=0$, then by the pigeon-hole principle, 
$$ |\Delta(E)| \ge \frac{q}{C}.$$ 

In view of (\ref{rudnev}), we conclude that $C \ge \frac{1}{c}$, so by taking $c$ to be arbitrarily small, we see, in a quantitative way, that as the size of $E$ drops below $q^{\frac{d+1}{2}}$, the estimate (\ref{finitefalconerestimate}) no longer holds for at least one $t$. By working with the example in (\cite{HIKR10}) directly, one can show that the same conclusion holds for every $t$. We leave the details to the interested reader. 

\vskip.125in 

\subsubsection{The parabolic case} In even dimensions, we are so far unable to produce an example showing that (\ref{finitefalconerestimate}) fails below the $|E| \approx q^{\frac{d+1}{2}}$ threshold in the spherical case. However, we are able to follow Valtr and produce an example in all dimensions when the sphere is replaced by a paraboloid 
$$ H=\{x \in {\Bbb F}_q^d: x_d=x_1^2+x_2^2+\dots+x_d^2 \}.$$ 

The fact that $H$ satisfies the Fourier estimate (\ref{salem}) is a straightforward consequence of the classical Gauss sum bounds. See, for example, \cite{LN97}. Thus if $|E|$ is much greater than $q^{\frac{d+1}{2}}$, then (\ref{finiteBfalconerestimate}) holds. 

To construct a sharpness example, let $q$ be a prime, so that ${\Bbb F}_q={\Bbb Z}_q$, the cyclic group on $q$ letters. We shall view the elements of ${\Bbb Z}_q$ via their representatives as 
$$ \{(0,1,2, \dots, q-1\}.$$ 

Let $A \subset {\Bbb Z}_q$ be given by 
$$ \{0,1,2, \dots, q^{\frac{1}{2}-\delta}\},$$ 

\vskip.125in 

where $\delta$ is a tiny fixed positive number and let ${\mathcal A} \subset {\Bbb Z}_q$ be given by 
$$ \{0,1,2, \dots, (d-1)q^{1-2\delta} \}.$$ 

Let $E$ be the Cartesian product of $(d-1)$ copies of $A$ and one copy of $\mathcal{A}$, or
$$E=A \times A \times \dots \times A \times {\mathcal A}.$$ 

Observe that 
$$ |E| \approx q^{\frac{d+1}{2}} q^{-(d+1)\delta}.$$ 

Then 
$$ |\{(x,y) \in E \times E: x-y \in H \}| \approx {|A|}^{d-1} \cdot {|A|}^{d-1} \cdot |{\mathcal A}|.$$ 

Observe that the right hand side of this expression divided by $\frac{{|E|}^2}{q}$ is $\approx q^{2 \delta}$, which can be made arbitrarily large. This completes the construction. 

\vskip.125in 

\section{Concluding remarks} 

\vskip.125in 

The purpose of this section is to summarize our findings and emphasize the questions yet to be answered. We have examined three manifestations of the incidence phenomenon, in discrete, continuous and arithmetic settings. The first problem stems from the celebrated Szemer\'edi-Trotter incidence theorem which we state in the following form. See, for example, \cite{Sz97}. Let $\Gamma$ be a strictly convex curve in the plane and let $P$, $L$ denote sets each containing $N$ points in the plane. Then 
\begin{equation} \label{stgeneral} |\{(x,y) \in P \times L: x-y \in \Gamma \}| \lesssim N^{\frac{4}{3}}. \end{equation}

Here $| \cdot |$ denotes the cardinality of a finite set. Valtr's construction, described above, shows that (\ref{stgeneral}) is, in general, sharp, just as it is in the case of incidences between points and lines. However, Valtr's example is for a particular $\Gamma$ only, namely for the parabola. 
\begin{problem} For which curves $\Gamma$ is the estimate (\ref{stgeneral}) sharp for? The Erd\H os single distance conjecture says that (\ref{stgeneral}) holds with $N^{\frac{4}{3}}$ on the right hand side replaced by $N \log(N)$ if $\Gamma$ is the unit circle. Even on a conjectural level, it would useful to attempt to understand the difference between the circle and the parabola that makes Valtr's example possible. \end{problem} 

Another problem examined in this paper is the Falconer estimate (\ref{Bfalconerestimate}). Mattila's example implies that this estimate is sharp for every symmetric convex body $B$ in the plane with a smooth boundary and everywhere non-vanishing curvature. In this paper, we extend his example to three dimensions. However, in dimensions four and higher, we currently do not have a universal sharpness example, only the one for a particular convex body given by Theorem \ref{main}.

\begin{problem} Does there exist a universal sharpness example for the estimate 
(\ref{Bfalconerestimate})? More precisely, given any centrally symmetric convex body $B$ in 
${\Bbb R}^d$, $d \ge 4$, with a smooth boundary and non-vanishing Gaussian curvature, can one show that $I_s(\mu)<\infty$ does not imply that (\ref{Bfalconerestimate}) holds? 
\end{problem} 

The third problem considered in this paper is the finite field analog of (\ref{Bfalconerestimate}) given by (\ref{finiteBfalconerestimate}). We have a sharpness example in the case $\Gamma=S_t$, a sphere, in odd dimensions and a sharpness example in the case $\Gamma$ is the paraboloid in all dimensions. Once again, the question is whether there is a universal sharpness example. 

\begin{problem} Does there exist a universal sharpness example for the finite field analog of Falconer's estimate (\ref{Bfalconerestimate})? More precisely, given any $\Gamma \subset {\Bbb F}_q^d$, $d \ge 2$, such that 

$$ |\Gamma|=q^{d-1}(1+o(1)) \ \text{and} \ |\widehat{\Gamma}(m)| \lesssim q^{-\frac{d+1}{2}}$$ 

\vskip.125in 

for any $m \not=(0, \dots, 0)$, does there exist $E \subset {\Bbb F}_q^d$ of size $|E|=cq^{\frac{d+1}{2}}$ such that 
$$ |\{(x,y) \in E \times E: x-y \in \Gamma \}| \ge \lambda(c) \frac{{|E|}^2}{q},$$ where $\lambda(c) \to \infty$ as $c \to 0$. In particular, is there a sharpness example in even dimension in the case when $\Gamma$ is a sphere? 
\end{problem} 

We have pointed out in Section \ref{discretevscontinuous} that while the Lenz example rules out a general non-trivial incidence theorem for points and translates of a fixed sphere in dimensions four and higher, interesting results are possible if a structural condition ($s$-adaptability) is imposed as was established in \cite{IJL09} in the case of well-distributed point sets. It should be possible to adapt the techniques of (\cite{IJL09} to resolve the following question. 

\begin{problem} To show that if a set $P$ containing $N$ points in the unit cube is $s$-adaptable (see Definition \ref{adaptable} above) for $s>\frac{d+1}{2}$, then 
\begin{equation} \label{sincidence} |\{(x,y) \in P \times P: {||x-y||})B=t \}| \lesssim N^{2-\frac{1}{s}} \end{equation} as long as $B$ is convex symmetric and the boundary of $B$ is smooth and has everywhere non-vanishing Gaussian curvature. Here, $| \cdot |$ denotes the cardinality of a finite set. 
\end{problem} 

Moreover, it would be interesting to bridge the gap between (\ref{sincidence}) and the Lenz example by proving an incidence bounds that depends on the degree to which the condition of $s$-adaptability holds. 

As a final remark, we note that discrete incidence theorems proved using methods related to the Falconer estimate produce, as a by-product, incidence theorems for annuli and points, not just points and surfaces. Furthermore, incidence estimates obtained this way produce, in many ranges of thickness of annuli, the ``expected" number of incidence. To put it simply, as the thickness of the annular region increases, a transition is made from ``extremal" to ``expected" results. Finding the precise threshold for this transition should lead to some interesting results.


\newpage

\begin{thebibliography}{20}

\bibitem{AI04} G. Arutuynyants and A. Iosevich, {\it Falconer conjecture, spherical averages, and discrete analogs}, Towards a theory of geometric graphs, Contemporary Mathematics \textbf{342}, Janos Pach, Editor. 

\bibitem{B94} J. Bourgain, {\it Hausdorff dimension and distance sets} Israel. J. Math. \textbf{87} (1994), 193-201. 

\bibitem{BMP00} P. Brass, W. Moser and J Pach, {Research Problems in Discrete Geometry}, Springer (2000). 

\bibitem{CEHIT10} D. Covert, B. Erdogan, D. Hart, A. Iosevich and K. Taylor {\it Finite point configurations, uniform distribution, intersections of fractals, and number theoretic consequences}, (in preparation), (2010). 

\bibitem{Erd05} B. Erdo\~{g}an {\it A bilinear Fourier extension theorem and applications to the distance set problem} IMRN (2006).

\bibitem{Fal86} K. J. Falconer {\it On the Hausdorff dimensions of distance sets} Mathematika \textbf{32} (1986) 206-212.

\bibitem{FKW90} H. Furstenberg, Y. Katznelson, and B. Weiss, {\it Ergodic theory and configurations in sets of positive density} Mathematics of Ramsey theory, 184-198, Algorithms Combin., 5, Springer, Berlin, (1990).

\bibitem{HB97} D. R. Heath-Brown, {\it Lattice points in spheres}, Number Theory in Progress, de Gruyter, Berlin \textbf{2} (1997).

\bibitem{HIKR10} Derrick Hart, Alex Iosevich, Doowon Koh and Misha Rudnev, {\it Averages over hyperplanes, sum-product theory in vector spaces over finite fields and the Erd\H os-Falconer distance conjecture}, Transactions of the American Mathematical Society, (in press), (2010).

\bibitem{Hux90} M. N. Huxley, {\it Exponential Sums and Lattice Points}, Proc. London Math. Soc. \textbf{60}, 471-502, (1990).

\bibitem{Hux96} M. N. Huxley, {\it Area, Lattice Points, and Exponentials Sums}, London Mathematical Society Monographs New Series \textbf{13}, Oxford Univ. Press, (1996).

\bibitem{Hi05} S. Hofmann and A. Iosevich, {\it Circular averages and Falconer/Erdos distance conjecture in the plane for random metrics}, Proc. Amer. Mat. Soc. 133 (2005), 133-143.

\bibitem{IJL09} A. Iosevich, H. Jorati and I. Laba, {\it Geometric incidence theorems via Fourier analysis} Trans. Amer. Math. Soc. \textbf{361} (2009) 6595-6611.

\bibitem{IL05} A. Iosevich and I. Laba, {\it K-distance, Falconer conjecture, and discrete analogs}, Integers, Electronic Journal of Combinatorial Numebr Theory, Proceedings of the Integers Conference in honor of Tom Brown, (2005) 95-106.

\bibitem{IR07} A. Iosevich and M. Rudnev, {\it On distance measures for well-distributed sets}, Journal of Discrete and Computational Geometry, \textbf{38}, (2007).

\bibitem{IR07ii} A. Iosevich and M. Rudnev, {\it Erdos distance problem in vector spaces over finite fields}, Transactions of the American Mathematical Society, (2007).

\bibitem{IR09} A. Iosevich and M. Rudnev, {\it Freiman's theorem, Fourier transform, and additive structure of measures}, Journal of the Australian Mathematical Society, \textbf{86}, (2009), 97-109.

\bibitem{IRU09} A. Iosevich, M. Rudnev, I. Uriarte-Tuero, {\it Theory of dimension for large discrete sets and applications}, http://arxiv.org/pdf/0707.1322 (submitted for publication) (2009). 

\bibitem{K77} S. Konyagin, {\it Integral points on strictly convex closed curves}. Mat. Zametki 1977, \textbf{21} (issue 6), 799-806.

\bibitem{LN97} R. Lidl and H. Niederreiter, {\it Finite fields}, Cambridge Univ. Press (1997).

\bibitem{Mat85} P. Mattila, {\it On the Hausdorff dimension and capacities of intersections}, Mathematika \textbf{32} (1985) 213-217. 

\bibitem{Mat87} P. Mattila, {\it Spherical averages of Fourier transforms of measures with finite energy: dimensions of intersections and distance sets} Mathematika, \textbf{34} (1987),  207-228.

\bibitem{M95} P. Mattila, {\it Geometry of sets and measures in Euclidean spaces}, Cambridge University Press, \text{volume} 44, (1995). 

\bibitem{MS99} P. Mattila, P. Sjšlin, {\it Regularity of distance measures and sets} Math. Nachr. \textbf{204} (1999), 157-162.

\bibitem{Sa32} H, Sali\'e, U\"ber die Kloostermanschen summen S(u,v;q), Math.Z.34, 91-109 (1932).

\bibitem{SV04} J. Solymosi, V. Vu, {\it Distinct distances in high dimensional homogeneous sets} in Towards a Theory of Geometric Graphs (J. Pach, ed.), Contemporary Mathematics, \textbf{342}, Amer. Math. Soc., (2004). 

\bibitem{Sz97} L. Sz\'ekely, {\it A. Crossing numbers and hard Erd\H os problems in discrete geometry} Combin. Probab. Comput. \textbf{6} (1997), 353-358.

\bibitem{SST84} J. Spencer, E. Szemer\'edi, and W. T. Trotter. {\it Unit distances in the Euclidean plane} B. Bollob\'as, editor, ``Graph Theory and Combinatorics," pages 293–303. Academic Press, New York, NY, 1984.

\bibitem{V05} P. Valtr, {\it Strictly convex norms allowing many unit distances and related touching questions}, manuscript 2005.

\bibitem{W99} T. Wolff, {\it Decay of circular means of Fourier transforms of measures}, International Mathematics Research Notices \textbf{10} (1999) 547-567.

\bibitem{We48} A. Weil, {\it On some exponential sums,} Proc. Nat.
Acad. Sci. U.S.A. \textbf{34} (1948), 204--207.



\end{thebibliography}
\end{document}